%% file: Async_DMPC.tex
\documentclass[a4paper,fleqn]{cas-dc}

\usepackage[authoryear,longnamesfirst]{natbib}

\input{commands.tex}

\begin{document}

\let\WriteBookmarks\relax

% Short title and author names for running headers
\shorttitle{Asynchronous Sensitivity-Based Distributed NMPC}
\shortauthors{Pierer von Esch et al.}

% ----------------------------------------------------------------
% Title and authors
% ----------------------------------------------------------------
\title[mode=title]{Asynchronous Sensitivity-Based Distributed NMPC}

\author[1]{Maximilian Pierer von Esch}
\ead{maximilian.pierer@fau.de}

\author[1]{Andreas V\"olz}
\ead{andreas.voelz@fau.de}

\author[1]{Knut Graichen}
\ead{knut.graichen@fau.de}

\affiliation[1]{
    organization={Chair of Automatic Control,
                  Friedrich-Alexander-Universit\"at
                  Erlangen--N\"urnberg (FAU)},
    city={Erlangen},
    postcode={91058},
    country={Germany}
}

% Corresponding author
\cormark[1]
\cortext[1]{Corresponding author}

% ----------------------------------------------------------------
% Abstract
% ----------------------------------------------------------------
\begin{abstract}
This paper presents a cooperative distributed model predictive control (MPC) scheme for nonlinear continuous-time systems. The centralized optimal control problem is solved asynchronously via a fixed number of sensitivity-based distributed programming (SBDP) iterations. The proposed scheme requires only neighbor-to-neighbor communication and no synchronization between agents during optimization. Under nominal MPC stability and bounded information delay, local exponential stability is established for a sufficiently large number of per-agent SBDP iterations. Numerical and hardware-in-the-loop results on both Ethernet and Wi-Fi demonstrate the benefits of an asynchronous execution, reducing execution times by over $60\,\%$ while maintaining comparable closed loop performance.
\end{abstract}

% ----------------------------------------------------------------
% Keywords
% ----------------------------------------------------------------

\maketitle

% ----------------------------------------------------------------
% Main sections
% ----------------------------------------------------------------
\input{content/introduction.tex}
\input{content/problem_statement.tex}
\input{content/asynchronous_dmpc.tex}
\input{content/stability_analysis.tex}
\input{content/validation.tex}

\input{content/conclusions.tex}

% ----------------------------------------------------------------
% Author contributions, if required
% ----------------------------------------------------------------
% Add CRediT contributions here if requested by the journal.
% \printcredits

% ----------------------------------------------------------------
% Bibliography
% ----------------------------------------------------------------
\bibliographystyle{cas-model2-names}
\bibliography{async_DMPC}

\end{document}

%% file: commands.tex
\usepackage[utf8]{inputenc}
\usepackage{amsmath,amsfonts,amssymb}
\usepackage{dsfont}
\usepackage{algorithm}
\usepackage{algpseudocode}
\usepackage[per-mode=fraction]{siunitx}
\usepackage{tikz}
\usepackage{tabularx}
\usepackage{accents}
\usepackage{pgfplots}
\usepackage{siunitx}
\usepackage{url}
\pgfplotsset{compat=1.18}
\usepgfplotslibrary{groupplots}
\usepackage[T1]{fontenc}      % use T1 fonts (so accented letters are printable)
\usepackage[utf8]{inputenc}   % source file is UTF‑8

\newcommand*{\vm}[1]{\boldsymbol{#1}}
\newcommand*{\trans}{^{\top}}
\newcommand*{\dd}{\mathrm{d}}

\newcommand*{\neighs}{j \in \mathcal{N}_i}
\newcommand*{\neighsi}{j \in \mathcal{\bar N}_i}

\newcommand*{\agents}{i \in \mathcal{V}}
\newcommand*{\Ni}[1]{\vm{#1}_{\mathcal N_i}}

\newcommand*{\inds}{^{*}}

\newcommand*{\indq}{^{q}}

\newcommand*{\indqin}{^{q_i+1}}
\newcommand*{\indqjn}{^{q_j+1}}
\newcommand*{\indqi}{^{q_i}}
\newcommand*{\indqj}{^{q_j}}

\newcommand*{\delayindex}{^{q - d_{ij}(q)}}

\newcommand{\nablaxi}{\nabla_{\vm x_i}}

\newcommand*{\floor}[1]{\lfloor #1 \rfloor}
\newcommand{\contr}{C^{\floor{\frac{q}{s+d}}}}

\newcommand*{\st}{\operatorname{s.\!t.}}

\newcommand*{\indqn}{^{q+1}}

\newcommand*{\timetau}{\tau \in [0,T]}
\newcommand*{\timetausample}{\tau \in [0,\Delta t)}

\newcommand*{\thisstate}{\vm \xi }

\newtheorem{theorem}{Theorem}
\newtheorem{lemma}{Lemma}

\newtheorem{assumption}{Assumption}

\newtheorem{proof}{Proof}

\DeclareMathOperator*{\spec}{\mathrm{spec}}

\definecolor{mycolor1}{rgb}{0.00000,0.44700,0.74100} %blue 
\definecolor{mycolor2}{rgb}{0.85000,0.32500,0.09800} %red 

\pgfplotscreateplotcyclelist{histogramm}{ % Reihenfolge für Linienstyles
	{mycolor1,thick},
	{mycolor2,thick},
}

%% file: content/introduction.tex
\section{Introduction}
The distributed operation of large-scale interconnected systems is a central challenge in modern cyber-physical applications, including energy systems \cite{Worthmann}, process engineering \cite{Liu}, and autonomous mobile robots \cite{Rosenfelder}. Model predictive control (MPC) provides a systematic framework for their control through the repeated solution of an optimal control problem (OCP) which accounts for coupled system dynamics and constraints. However, the centralized solution of the resulting high-dimensional OCPs quickly becomes intractable in real-time MPC applications. Distributed model predictive control (DMPC) addresses this issue by exploiting the structure of large-scale systems and distributing the computational effort among multiple agents. In particular, cooperative DMPC seeks to retain the stability properties and performance of centralized MPC while preserving the scalability and modularity of distributed implementations. In contrast, non-cooperative schemes require less communication but usually do not achieve the same quality as centralized control \cite{Muller}. For cooperative DMPC, a variety of distributed optimization algorithms have been utilized, such as dual decomposition \cite{Giselsson}, the alternating direction method of multipliers (ADMM) \cite{Boyd}, the augmented Lagrangian alternating direction inexact Newton method (ALADIN) \cite{Houska}, decentralized sequential quadratic programming (d-SQP) \cite{Stomberg}, and  sensitivity-based distributed programming (SBDP) \cite{Pierer2}.
However, most existing cooperative DMPC schemes using the aforementioned algorithms implicitly rely on a synchronous execution, requiring strict coordination between agents and introducing idle times due to communication delays and heterogeneous computational loads \cite{Bestler,Stomberg}. 
To address these limitations, asynchronous distributed optimization methods have recently gained attention, allowing agents to update using locally available and potentially outdated information \cite{Bertsekas,Pierer2}. Examples commonly used in DMPC include asynchronous (also called flexibel) variants of ADMM \cite{Hong}, ALADIN \cite{Du} or dual decomposition \cite{Su}. Usually, these investigation involve static optimization problems with a notable exception of \cite{Pierer3}. While this reduces waiting times and synchronization overhead, it also introduces additional challenges in terms of convergence and closed loop stability, particularly in real-time MPC implementations where only a finite number of iterations can be performed per sampling step. While there exists work on utilizing asynchronous distributed optimization, e.g., ADMM or d-SQP, within experimental DMPC setups \cite{Stomberg2,Burk} and showcasing their practical viability, theoretical guarantees for asynchronous optimization algorithms applied to nonconvex OCPs within real-time cooperative DMPC schemes remain largely unaddressed.

This paper aims at closing this gap and develops a cooperative DMPC scheme where a centralized OCP is solved in an asynchronous fashion via SBDP. Asynchronism is restricted to the optimization layer, i.e., SBDP iterations are performed asynchronously within each MPC step while the MPC steps are synchronized between agents. This, for example, formally enables the use of best-effort UDP-based communication protocols instead of slower connection-oriented protocols like TCP, representing an important step toward real-time DMPC.  Building on R-linear convergence properties of the asynchronous execution of SBDP \cite{Pierer2}, we derive conditions on the number of per-agent iterations to ensure local exponential stability of the closed loop system despite inexact optimization. The standing assumptions are that the centralized MPC formulation is stabilizing and that the delay of the information in the network is upper bounded. The stability analysis is inspired by recent suboptimal (D)MPC investigations, including MPC frameworks with a fixed number of optimizer iterations per sampling step \cite{Graichen,Liao-McPherson}, the class of real-time iterations (RTI) \cite{Diehl,Zanelli}, and stability analyses for synchronous DMPC \cite{Bestler,Pierer,Stomberg}. In contrast to these works, we consider asynchronous distributed optimization with an R-linearly convergent algorithm in a continuous-time DMPC framework. Moreover, the optimization error is not required to decrease monotonically, rendering standard RTI arguments based on Q-linear convergence not directly applicable.

The rest of the paper is structured as follows: Section \ref{sec:problem_statement} introduces the problem statement, while Section \ref{sec:async_DMPC} developes the DMPC scheme. Afterward, stability of the closed loop is analyzed in Section \ref{sec:stability_analysis}, before a simulative and hardware-in-the-loop validation is given in Section~\ref{sec:validation}. Concluding remarks are made in Section~\ref{sec:conclusion}.

Notation: The stacking of vectors $\vm v_i \in \mathbb{R}^{n_i}$, $\agents$ from a set $\mathcal{V} \subset \mathbb{N}$ is denoted as $[\vm v_i]_{\agents} = [\vm v_1\trans, \dots, \vm v_{|\mathcal{V}|}\trans]\trans$, while for $\vm g_i(\cdot) \in C([0,T];\mathbb{R}^{n_i})$ it is $ \vm g(\cdot)=(\vm g_i(\cdot))_{\agents}$ with $\vm g(t)= [ \vm g_1(t)\trans,\dots, \vm g_{|\mathcal{V}|}(t)\trans]\trans$, $\forall t \in [0,T]$. The block-maximum norm of a vector $[\vm v_i]_{\agents} \in \mathbb{R}^n$ is given as $\|\vm v \|_b = \max_{\agents} \|\vm v_i\|_2$, while for vector functions $\vm g(\cdot) \in C([0,T];\mathbb{R}^n)$ we use the supremum norm $\|\vm g(\cdot)\|_{b,\infty} = \sup_{t \in [0,T]}\|\vm g(t)\|_b$. An open neighborhood of a point $\vm v_0 \in \mathbb{R}^n$ is defined as $\mathcal{B}_r(\vm v_0):=\{\vm v\in \mathbb{R}^n\, |\, \| \vm v - \vm v_0\|_b<r\}$, while for $ \vm g_0(\cdot) \in C([0,T];\mathbb{R}^n)$ it is given as $\mathcal{B}_r(\vm g_0(\cdot)):=\{ \vm g(\cdot) \in C([0,T];\mathbb{R}^n)\,|\, \| \vm g(\cdot) - \vm g_0(\cdot)\|_{b,\infty}<r\}$. 

%% file: content/problem_statement.tex
\section{Problem statement} \label{sec:problem_statement}
We consider multi-agent systems described by an undirected, connected graph $\mathcal{G}(\mathcal{V}, \mathcal{E})$ where each vertex $ i \in \mathcal{V}=\{1,\dots,M\}$ represents an agent and the edge set $\mathcal{E} \subset \mathcal{V} \times \mathcal{V}$ characterizes the coupling structure between subsystems. For each agent $i$, the open neighborhood is defined as $\mathcal{N}_i := \{j \in \mathcal{V}\setminus\{i\}\, |\, (i,j) \in \mathcal{E}\}$ and contains all subsystems directly coupled to agent $i$, while the closed neighborhood is denoted as $\mathcal{\bar N}_i := \mathcal{N}_i \cup \{i\}$. We consider a DMPC scheme for setpoint stabilization where the agents cooperatively solve the centralized OCP 
\begin{subequations} \label{eq:central_OCP}
\begin{alignat}{2}
    \min_{\vm {\bar u}(\cdot)} \quad 
    & \sum_{\agents} J_i(\vm {\bar u}_i(\cdot); (\vm {\bar x}_j(\cdot))_{\neighs})  \\
\st \quad
    & \vm {\dot{\bar x}}_i (\tau) = \vm f_{i}( \vm {\bar x}_i(\tau), \vm {\bar u}_i(\tau), [\vm {\bar x}_j(\tau)]_{\neighs})\label{eq:central_OCP_dynamics}\\
    & \vm {\bar x}_i(0) = \vm \xi_i \\
    & \vm{\bar u}_i(\tau) \in \mathbb{U}_i\,, \quad \timetau\,, \quad \forall\agents
\end{alignat}
\end{subequations}
in a distributed and possibly asynchronous fashion at each sampling time $t_k = t_0 + k \Delta t$, $k \in \mathbb{N}_0$ for a fixed sampling period $\Delta t>0$. The state and input of each agent are denoted as $\vm x_i(t) \in \mathbb{R}^{n_i}$ and $\vm u_i(t) \in \mathbb{R}^{m_i}$, respectively. The bar notation denotes optimization variables depending on the prediction time $\timetau$, in contrast to system variables depending on system time $t>0$. OCP \eqref{eq:central_OCP} is parametrized by the initial condition $ \vm \xi_i \in \mathbb{R}^{n_i}$, which is instantiated at each sampling time $ t = t_k$ as $ \vm \xi_i = \vm x_{i,k}= \vm{{x}}_i(t_k)$. The local constraint set $\mathbb{U}_i \subset \mathbb{R}^{m_i}$ is supposed to be nonempty, compact and convex. The local cost functional is given as
\begin{equation} \label{eq:agent_costs}
    J_i \!= L_{i}( \vm {\bar x}_i(T)) +\! \int_{0}^{T}\! l_i(\vm {\bar x}_i(\tau), \vm {\bar u}_i(\tau), [\vm {\bar x}_j(\tau)]_{\neighs})\, \dd\tau 
\end{equation}
with prediction horizon $T\geq \Delta t$. Furthermore, we presuppose a neighbor-affine structure for costs \eqref{eq:agent_costs} and dynamics \eqref{eq:central_OCP_dynamics}, i.e., $\vm f_i = \vm f_{ii}(\vm x_i, \vm u_i) + \sum_{\neighs} \vm f_{ij}(\vm x_i, \vm x_j)$ and $l_i = l_{ii}(\vm x_i, \vm u_i) + \sum_{\neighs} l_{ij}(\vm x_i, \vm x_j)$. 
The centralized quantities are obtained by stacking all local variables as $\vm x(t) = [\vm x_i(t)]_{\agents} \in \mathbb{R}^n$, $\vm u(t)= [\vm u_i(t)]_{\agents} \in \mathbb{R}^m$, $ \vm \xi = [\vm \xi_i]_{\agents} \in \mathbb{R}^n$, $ \vm x_k = [\vm x_{i,k}]_{\agents} \in \mathbb{R}^n$, and $\mathbb{U} = \Pi_{\agents }\mathbb{U}_i \subset \mathbb{R}^m$, with $n= \sum_{\agents} n_i$ and $m = \sum_{\agents} m_i$. The subsequent stability analysis is based on an equivalent centralized representation of OCP~\eqref{eq:central_OCP} 
\begin{subequations} \label{eq:ocp}
\begin{align}
    V(\thisstate) = \min_{\vm {\bar u}(\cdot)} \quad 
    & L(\vm {\bar x}(T)) + \int_{0}^{T} l(\vm {\bar x}(\tau), \vm {\bar u}(\tau))\,\dd \tau  \label{eq:costs_ocp} \\
\st \quad
    & \vm {\dot{ \bar x}} (\tau) = \vm f( \vm {\bar x}(\tau), \vm {\bar u}(\tau))\,,\, \vm {\bar x}(0) = \thisstate \label{eq:dynamics_ocp} \\
    & \vm {\bar u}(\tau) \in \mathbb{U}\,, \quad \timetau
\end{align}
\end{subequations}
which is obtained by appropriately stacking the local quantities. Specifically, the centralized dynamics are given as $ \vm f: \mathbb{R}^{n} \times \mathbb{R}^{m} \rightarrow \mathbb{R}^{n}$ with $\vm f(\vm x, \vm u) := [\vm f_i(\vm x_i, \vm u_i, [\vm x_j]_{\neighs})]_{\agents} $ and are supposed to be at least twice continuously differentiable. Without loss of generality the origin is an equilibrium of the system for $\vm u=\vm 0$, i.e., $\vm f(\vm 0, \vm 0) = \vm 0$.
The objective \eqref{eq:costs_ocp} consists of the terminal costs $L : \mathbb{R}^{n} \rightarrow \mathbb{R}$  and integral costs $l : \mathbb{R}^{n} \times \mathbb{R}^{m} \rightarrow \mathbb{R}$ with  $ L(\vm x) := [L_i(\vm x_i)]_{\agents} $ as well as $ l(\vm x, \vm u) := [l_i(\vm x_i, \vm u_i, [\vm x_j]_{\neighs})]_{\agents} $ and are supposed to be at least twice continuously differentiable. Furthermore, we assume $\vm {\bar u}(\cdot)\in C([0,T];\mathbb U)$. For a time $\timetau$, initial condition  $\vm {\bar x}(0)=\thisstate$, and input trajectory $\vm {\bar  u}(\cdot)$, define $\vm \Psi: [0,T]\times \mathbb R^n\times C([0,T];\mathbb U) \rightarrow \mathbb R^n $ as the solution map of \eqref{eq:dynamics_ocp}, i.e., $ \vm \Psi(\tau;\thisstate, \vm {\bar  u}(\cdot)) := \vm x(\tau;\thisstate,\vm {\bar u}(\cdot))$. We require the existence of optimal control trajectories $\vm { \bar u}\inds(\cdot; \thisstate)$ of the central OCP \eqref{eq:ocp} and the corresponding state trajectory $\vm {\bar x}\inds(\cdot; \thisstate) := \vm \Psi(\cdot;\thisstate,\vm {\bar u}\inds(\cdot; \vm \xi))$. Specifically, we assume the existence of a compact set $\mathbb{X}_0 \subset \mathbb{R}^n$ with $\vm 0 \in \mathbb{X}_0$ and a set $ \mathbb{T}:=(0, \hat T]$ with some $\hat T \in \mathbb{R}_{>\Delta t}$ such that for all $\thisstate \in \mathbb{X}_0$ and $T \in \mathbb{T}$, the OCP~\eqref{eq:ocp} has a locally unique minimizer $(\vm {\bar u}\inds(\cdot; \thisstate),\vm {\bar x}\inds(\cdot;\thisstate))$.
The value function of OCP \eqref{eq:ocp} is denoted as $V: \mathbb{X}_0 \rightarrow \mathbb{R}$.

Standard MPC schemes assume that the optimal solution is available at each sampling time $t_k$. Then, the initial segment of $\vm {\bar u}\inds(\tau; \vm x_k)$ is applied to the system, i.e.,
\begin{equation} \label{eq:MPC_control_law}
    \vm u(t_k + \tau) = \vm {\bar u}\inds(\tau; \vm x_k)\,, \quad \timetausample\,.
\end{equation}
At the next sampling instant $t_{k+1}$, OCP \eqref{eq:ocp} is resolved using the updated initial condition which, in the nominal case, is given by $\vm x\inds_{k+1} = \vm \Psi(\Delta t;\vm x_k,\vm {\bar u}\inds(\cdot; \vm x_k))$. 
In multi-agent systems with many agents, it increasingly becomes intractable to solve the OCP \eqref{eq:central_OCP} in a centralized fashion. Therefore, the remainder of this paper is concerned with the distributed and asynchronous solution of OCP \eqref{eq:central_OCP} and the stability analysis of the closed loop resulting from the proposed DMPC control law resulting from the asynchronous optimization process. 

%% file: content/asynchronous_dmpc.tex
\section{Asynchronous sensitivity-based DMPC} \label{sec:async_DMPC}
This section presents the distributed solution of OCP~\eqref{eq:central_OCP} via asynchronous SBDP (a-SBDP) within a real-time DMPC framework. The algorithm exploits sensitivity information between neighboring agents to augment the local cost functional and is based on primal decomposition to achieve a coordinated decoupling between agents. 
\subsection{Asynchronous SBDP and DMPC control law}
Each agent $\agents$ constructs a local decoupled OCP, which is solved iteratively in each per-agent SBDP iteration $q_i = 0,1,\dots$ 
\begin{subequations}\label{eq:local_OCP_alg_async}
\begin{alignat}{2}
                        \min_{\vm {\bar u}_i(\cdot)} \quad&  \bar J_i\indqi(\vm {\bar u}_i(\cdot); \vm \xi_i):= J_i(\vm {\bar u}_i(\cdot); (\vm { \bar x}_j\indqj(\cdot))_{\neighs}, \vm \xi_i) \label{eq:local_OCPalg_cost_async}\\
+\sum_{\neighs}& \int_0^T \vm g_{ji}(\vm {\bar x}_i\indqi(\tau), \vm {\bar x}_j\indqj(\tau), \vm {\bar \lambda}_j\indqj(\tau))\trans \delta \vm {\bar x}_i(\tau)\,\dd \tau \nonumber      \\               
                      ~\st \quad&\vm{\dot {\bar x}}_i(\tau) =\vm f_{i}( \vm {\bar x}_i(\tau), \vm {\bar u}_i(\tau), [\vm {\bar x}_j\indqj(\tau)]_{\neighs}) \label{eq:local_OCPalg_dyn_async}\\
                  & \vm {\bar x}_i(0) = \vm \xi_i,\, \quad \vm {\bar u}_i(\tau)\in \mathbb{U}_i\,,\quad \timetau \label{eq:local_OCPalg_constraint_async}
\end{alignat}
\end{subequations}
with $\delta \vm {\bar x}_i(\tau) = \vm {\bar x}_i(\tau) - \vm {\bar x}_i\indqi(\tau)$ and where $\vm g_{ji}(\cdot)$ is given by
\begin{equation}\label{eq:gradient}
    \vm g_{ji}(\vm x_i, \vm x_j, \vm \lambda_j) \!=\! \nablaxi l_{ji}(\vm x_j,\vm x_i) \!+\! \nablaxi f_{ji}(\vm x_j, \vm x_i)\trans \vm \lambda_j
\end{equation}
and can be interpreted as the sensitivity of the cost functional $J_j(\cdot)$ w.r.t.\ a change in the state trajectory $\vm {\bar x}_i(\cdot)$, $\agents$. In \eqref{eq:gradient}, $\vm {\bar \lambda}_j(\tau) \in \mathbb{R}^{n_j}$ denotes the adjoint state of the neighbor associated with the local dynamics \eqref{eq:local_OCPalg_dyn_async} which is typically available from the solution of \eqref{eq:local_OCP_alg_async}. Otherwise, $\vm { \bar \lambda}_i\indqin(\cdot)$ may be computed by solving
\begin{equation} \label{eq:Hamiltonian}
    \vm {\dot {\bar \lambda}}_i\indqin\!(\tau) \!=\! - \nablaxi H_i\indqi(\vm {\bar x}_i\indqin(\tau), \vm {\bar u}_i\indqin(\tau), \vm {\bar \lambda}_i\indqin(\tau))
\end{equation}
with the local Hamiltonian $H_i\indqi := l_i(\vm x_i, \vm u_i, [\vm x_j\indqj]_{\neighs}) + \vm \lambda_i\trans \vm f_i(\vm x_i, \vm u_i,[\vm x_j\indqj]_{\neighs}) + \sum_{\neighs} \vm g_{ji}(\vm x_i\indqi, \vm x_j\indqj, \vm \lambda_j\indqj)\trans (\vm x_i - \vm x_i\indq) $ and terminal condition $\vm { \bar \lambda}_{i}(T)= \nablaxi L_i( \vm {\bar x}_i(T))$. Suppose $\vm p_i\indqi(\tau; \vm \xi_i) := [\vm {\bar x}_i\indqi(\tau; \vm \xi_i)\trans,\, \vm {\bar \lambda}_i\indqi(\tau;\vm \xi_i)\trans]\trans$, $\timetau$ denotes the iterate of agent $\agents$, then the asynchronous execution of SBDP from a local viewpoint is given in Algorithm~\ref{alg:SBDP_cont_local_async} \cite{Pierer}. Each iteration consists of solving OCP \eqref{eq:local_OCP_alg_async} followed by a communication step. Unlike the synchronous execution, agents do not wait for updated information from all neighbors and directly proceed with the currently available trajectories. This is reflected in the use of local iteration counters $q_i$, whereby agents may rely on outdated neighbor information indexed by $q_j$. Consequently, in the asynchronous case one generally has $q_i \neq q_j$, whereas in a synchronous execution the counters are enforced to satisfy $q_i = q_j$. 

In the context of a real-time DMPC scheme, Algorithm~\ref{alg:SBDP_cont_local_async} is used to solve the centralized OCP \eqref{eq:central_OCP} in an asynchronous fashion and is terminated after a fixed number of local iterations $\bar q_i$ have been executed.
The resulting approximate solution $\vm {\hat u}(\tau; \vm x_k) :=[{\vm { \bar u}}_i^{\bar q_i}(\tau; \vm x_{i,k})]_{\agents}$, $\timetau$, is used to define the DMPC control law
\begin{align} \label{eq:DMPC_control_law}
   {\vm u}(t_k+\tau)  = \vm {\hat u}(\tau; \vm x_k)\,, \quad \timetausample
\end{align}
which is computed locally by each agent. The centralized closed loop evolution under \eqref{eq:DMPC_control_law} is given as
$ \vm x_{k+1} = \vm \Psi(
\Delta t; \vm x_k,\vm {\hat u}(\cdot;\vm x_k))$.  Furthermore, Algorithm \ref{alg:SBDP_cont_local_async} is warm-started with the previous solution, i.e., 
$\vm p_i^0(\tau; \vm x_{i,k+1}) = \vm p_i^{\bar{q}_i}(\tau;\vm x_{i,k})$, $ \timetau$ which additionally can be shifted by $\Delta t$ with a suitable tail approximation if desired.
\begin{algorithm}[tb] 
		\caption{a-SBDP for solving OCP \eqref{eq:central_OCP} \cite{Pierer}}
		\begin{algorithmic}[1]
        \addtocounter{ALG@line}{-1}
			\State Initialize and send $\vm p_i^0(\tau; \vm \xi_i)$, $\timetau$ to all neighbors $\neighs$; Set $\bar q_i$, initial condition $\vm \xi_i$, and $q_i \leftarrow 0$
			\State Compute $(\vm {u}_i\indqin(\cdot), \vm p_i\indqin(\cdot))$ by solving \eqref{eq:local_OCP_alg_async}.
\State Broadcast $\vm p_i\indqin(\cdot)$ to all neighbors $\neighs$.
\State Receive $\vm p_j\indqjn(\cdot)$ from a subset of neighbors.
\State Stop if $q_i= \bar q_{i} - 1$ . Otherwise, go to line~$1$ with $q_i \leftarrow q_i+1$.
\end{algorithmic}\label{alg:SBDP_cont_local_async}
\end{algorithm}
\subsection{Characterization of the asynchronous process}
While Algorithm~\ref{alg:SBDP_cont_local_async} is directly suitable for implementation, a global description of the asynchronous execution is required for its convergence analysis. To this end, we introduce an abstract local solution operator associated with each agent.
For each agent $\agents$, we interpret \eqref{eq:local_OCP_alg_async} as a parametric OCP, where the parameters are given by the external trajectories $(\vm p_j(\cdot))_{\neighsi}$. Under the assumption that this OCP admits a locally unique solution, we define the solution operator $\vm \Phi_i: \Pi_{\neighsi} C([0,T]; \mathbb{R}^{p_j}) \rightarrow C([0;T],\mathbb{R}^{p_i})$ which maps the external trajectories to the corresponding local primal-dual solution trajectory of agent $i$, i.e.,
\begin{align}
    \vm \Phi_i((\vm p_j(\cdot))_{\neighsi}; \vm \xi_i) := \vm p_i(\cdot; (\vm p_j(\cdot))_{\neighsi}, \vm \xi_i)\,.
\end{align}
Equivalently, we concatenate  $(\vm \Phi_i((\vm p_j(\cdot))_{\neighsi}); \vm \xi_i)_{\agents} =\vm \Phi(\vm p(\cdot); \vm \xi) = (\vm {\bar x}(\cdot; \vm p(\cdot), \vm \xi), \vm {\bar \lambda}(\cdot;\vm p(\cdot), \vm \xi))$, $\vm p(\cdot) = (\vm p_i(\cdot))_{\agents}$ as the stacked solutions of the OCPs \eqref{eq:local_OCP_alg_async} and $\vm {\bar u}(\cdot;\vm p(\cdot), \vm \xi_i)$ as the stacked control. To describe the asynchronous process centrally, let $\mathcal{S}(q) \in \mathcal{V}$ denote the set of agents that perform an update during the global update event $q$, where $q=0,1,\dots$ indexes the sequence of update events across the network. Here, we consider the edge case where $\mathcal{S}(q) = \{a(q)\}$ is a singleton, i.e., each global update event corresponds to exactly one agent performing its local update, leading to
$ \bar q := \sum_{\agents} \bar q_i $ per-agent updates being performed globally within one MPC step.
Then, Algorithm~\ref{alg:SBDP_cont_local_async} is described by the recursion 
\begin{align}
  &\vm p_i\indqn(\cdot; \vm \xi_i)= \nonumber \\
  &= \begin{cases} \vm \Phi_i(({\vm p}_{j}\delayindex(\cdot; \vm \xi_j);\vm \xi_i)_{j \in \mathcal{\bar N}_i})\,,\!&i = a(q)\\
\vm p_i\indq(\cdot; \vm \xi_i)\,, \!&i \neq a(q)
\end{cases} 
\end{align}
with $d_{ij}(q)\in\mathbb N_0$ denoting the delay, measured in global per-agent updates, of the trajectory $\vm p_j(\cdot; \vm \xi_j)$, $j\in\mathcal N_i$, used by agent $\agents$ at update~$q$. Consequently, agent $\agents$ performs its update using trajectories generated $d_{ij}(q)$ update events earlier. The delays therefore quantify the age of the information available to agent $\agents$ and are generally unknown a priori. Such delays may arise from heterogeneous solution times of the local OCPs or  communication latencies during trajectory exchange.
Furthermore, we assume that either the initial trajectories $\vm p_i^0(\cdot)$  are transmitted without delay or that each agent initializes $\vm p_j^0(\cdot)$ with the corresponding values stored by neighboring agents $j\in\mathcal V$.  To ensure that the asynchronous process remains well defined and that information does not become arbitrarily outdated, the degree of asynchronism is restricted by the following assumption.
\begin{assumption} \label{ass:asynchronism}
    There exist finite integers $d \in \mathbb{N}_0$ and $s\in \mathbb{N}$ such that i) the delays are uniformly bounded, i.e., 
    \begin{align}
        0\leq d_{ij}(q)\leq d <\infty\,, \quad \forall\agents,\, \forall\neighs.
    \end{align}
    and ii) that every agent performs an update at least in every s consecutive iterations, i.e., 
    \begin{align}
       \{a(q)\} \!\cup\! \{a(q+1)\} \!\cup\! \dots \!\cup\! \{a(q+s-1)\} = \mathcal{V} \,.
    \end{align}
\end{assumption}
This assumption is often referred to as partial asynchronism in literature which is in contrast to total asynchronism where delays may become unbounded \cite{Bertsekas}. The first condition ensures that information used by any agent is at most $d$ update events old, while the second ensures that no agent permanently drops out of the process, e.g., hardware failure. A synchronous execution is recovered by setting $\mathcal{S}(q) = \mathcal{V}$ for all $q$ and enforcing $d_{ij}(q) = 0$, i.e., all agents update simultaneously at each event and no information gets lost, leading to $s=1$ and $d=0$ in Assumption \ref{ass:asynchronism}. In this case, the event index $q$ coincides with a classical synchronous iteration index.
\subsection{Algorithmic analysis}
In this section, we analyze the convergence of asynchronous SBDP as a prerequisite for the subsequent stability analysis of the closed loop system under the control law~\eqref{eq:DMPC_control_law}. In contrast to~\cite{Pierer}, however, we avoid convexity assumptions as well as assumptions directly imposed on solvability of the local OCPs~\eqref{eq:local_OCP_alg_async}. Instead, the properties are derived from assumptions imposed solely on the centralized solutions. Thus, the problem class is extended compared to the previous work~\cite{Pierer}. The analysis consists of establish existence and Lipschitz continuity of the local solution trajectories in Lemma \ref{lem:solvability_regularity_local_ocps} and then moves on to apply this result in Theorem \ref{thm:conv_SBDP} to establish R-linear convergence.

To establish existence and local uniqueness of the solution trajectories of OCP \eqref{eq:local_OCP_alg_async}, we employ a coercivity condition which mirrors the standard second-order sufficiency condition from nonlinear programming.
To formulate this condition, define $\vm A_i(\tau) = \nabla_{\vm x_i} \vm f_i\inds(\tau)$, $\vm B_i(\tau) = \nabla_{\vm u_i} \vm f_i\inds(\tau)$, and $\vm P_i(T) = \nabla_{\vm x_i \vm x_i}^2 L_i(\vm { \bar x}_i\inds(T))$, $\vm Q_i(\tau)=\nabla_{\vm x_i \vm x_i}^2 \bar H_i\inds(\tau)$, $\vm R_i(\tau)=\nabla_{\vm u_i \vm u}^2 \bar H_i\inds(\tau)$, $\vm S_i(\tau) =2 \nabla_{\vm x_i \vm u_i}^2 \bar H_i\inds(\tau)$, $\timetau$ where $\vm f_i\inds(\tau) := \vm f_i(\vm {\bar x}_i\inds(\tau; \vm \xi), \vm {\bar u}_i\inds(\tau; \vm \xi), \Ni{\bar x}\inds(\tau; \vm \xi))$ as well as $ \bar H_i\inds(\tau) :=  H_i\inds(\vm {\bar  x}_i\inds(\tau; \vm \xi), \vm { \bar u}_i\inds(\tau; \vm \xi), \vm { \bar \lambda}_i\inds(\tau; \vm \xi))$ denote the evaluation of dynamics \eqref{eq:central_OCP_dynamics} and Hamiltonian \eqref{eq:Hamiltonian} along the solution for each $\vm \xi \in \mathbb{X}_0$, cf.  \cite[Sec. 6]{Dontchev}.
\begin{assumption} \label{ass:coercivity}
For any $\agents$, $\vm \xi \in \mathbb{X}_0$, and $T \in \mathbb{T}$ there exists a constant $\rho_i \in \mathbb{R}_{>0}$ such that 
\begin{multline}
   \|\delta \vm x_i(T)\|_{\vm P_i(T)}^2+  \int_0^T\! \bigg( \|\delta \vm x_i(\tau)\|^2_{\vm Q_i(\tau)}+ \|\delta \vm u_i(\tau)\|^2_{\vm R_i(\tau)} \\ +\delta \vm x_i(\tau)\trans \vm S_i(\tau)\delta \vm u_i(\tau)\, \dd \tau \bigg) \geq\rho_i \! \int_0^T \!\!\| \delta \vm u_i(\tau)\|^2 \, \dd \tau
\end{multline}
holds for all perturbations $\delta \vm u_i(\cdot) \in C([0,T];\mathbb{U}_i)$ where $\delta \vm x_i(\tau)$ is the solution of the linearized dynamics $ \delta \vm {\dot x}_i(\tau) = \vm A_i(\tau) \delta \vm x_i(\tau) + \vm B_i(\tau) \delta \vm u_i(\tau)$ with $ \delta \vm x_i(0)= \vm 0$, $\timetau$.
\end{assumption}
The assumption neither implies convexity of the dynamics \eqref{eq:central_OCP_dynamics} nor of the costs \eqref{eq:agent_costs} and can be enforced algorithmically by adding regularization terms $ \frac{\gamma_i}{2}(\| \vm x_i - \vm x_i\indqi\|^2 + \| \vm u_i - \vm u_i\indqi\|^2)$ with sufficiency large penalty $\gamma_i \in \mathbb{R}_{\geq 0}$ to the integral costs in \eqref{eq:local_OCPalg_cost_async}. The next Lemma utilizes Assumption \ref{ass:coercivity} and applies a regularity result from parametric optimal control to the local OCPs, cf. \cite[Thm. 5]{Dontchev}.  
\begin{lemma} \label{lem:solvability_regularity_local_ocps}
Let Assumption \ref{ass:coercivity} hold. Then, for every $\vm \xi \in \mathbb{X}_0$ and $T \in \mathbb{T}$ there exist constants $r_1^\prime$, $r_1$, $L \in \mathbb{R}_{>0}$ such that for all $\vm p(\cdot)\in \mathcal{B}_{r_1}(\vm p\inds(\cdot; \vm \xi))$ the local OCPs~\eqref{eq:local_OCP_alg_async} admit a locally unique minimizer $(\vm { \bar u}(\cdot;\vm p(\cdot)), \vm { \bar  x}(\cdot;\vm p(\cdot)) \in \mathcal{B}_{r_1^\prime}((\vm u\inds(\cdot;\vm \xi), \vm x\inds(\cdot;\vm \xi)))$. Moreover, the associated adjoint trajectory $\vm {\bar \lambda}(\cdot; \vm p(\cdot))$ is bounded and locally unique and the mapping $\vm \Phi(\cdot)$ is Lipschitz continuous in the sense of
\begin{multline} \label{eq:Lipschitz_local_solution}
    \|\vm{\bar u}(\cdot; \vm p(\cdot)) - \vm{\bar u}\inds(\cdot; \vm \xi)\|_{b,\infty} +  \|\vm{\bar x}(\cdot; \vm p(\cdot)) - \vm{\bar x}\inds(\cdot;\vm \xi)\|_{b,\infty} \\  \!+\!\!\|\vm{\bar \lambda}(\cdot; \vm p(\cdot))\! -\!\! \vm{\bar \lambda}\inds(\cdot;\vm \xi)\|_{b,\infty}\!\leq\! L\|\vm p(\cdot) \! - \! \vm p\inds(\cdot;\vm \xi)\|_{b,\infty}\,.
\end{multline}  
\end{lemma}
\begin{proof}
At first, we observe that for $\vm p(\cdot) = \vm p\inds(\cdot;\vm \xi)$ we get that $\vm{\bar x}(\tau; \vm p\inds(\cdot; \vm \xi), \vm \xi) = \vm{\bar x}\inds(\tau;\vm \xi) $,  $\vm{\bar u}(\tau; \vm p\inds(\cdot;\vm \xi), \vm \xi) = \vm u\inds(\tau; \vm \xi) $ and  $\vm{\bar \lambda}(\tau; \vm p\inds(\cdot; \vm \xi), \vm \xi) = \vm{\bar \lambda} \inds(\tau; \vm \xi) $ are solutions of the concatenated local OCPs \eqref{eq:local_OCP_alg_async} since their optimality conditions are structurally identical to OCP \eqref{eq:ocp}. 
The result then follows by a direct application of \cite[Thm. 5]{Dontchev}. Although the theorem is stated for the functional setting $\vm {\bar u}(\cdot) \in L^\infty$ and $\vm {\bar x}(\cdot), \vm { \bar \lambda}(\cdot) \in W^{1,\theta}$, $\theta \in [1,\infty]$, it also applies in the present setting since, on the bounded interval $[0,T]$, the continuous  embeddings $C\subset L_\infty$ and $C \subset W^{1,\theta}$ hold for all $\theta \in [1,\infty]$. Moreover, all assumptions of \cite[Thm. 5]{Dontchev} are satisfied: i) the constraint set $\mathbb{U}$ is nonempty, closed and convex; ii) all involved functions in the local OCPs \eqref{eq:local_OCP_alg_async} are twice continuously differentiable; iii) the coercivity Assumption \ref{ass:coercivity} holds; and iv) we assume existence and local uniqueness of the optimal solutions. Finally, 
the uniqueness and boundedness of $\vm {\bar \lambda}(\cdot;\vm p(\cdot))$ follows from the uniqueness of $(\vm {\bar u}(\cdot, \vm {\bar p}(\cdot)),\vm {\bar x}(\cdot; \vm p(\cdot))$ and the adjoint dynamics whose right-hand-side is linear in $\vm \lambda$ and continuous w.r.t.\ $\tau$ for all $\timetau$ such that standard arguments imply the existence, uniqueness and boundedness of $\vm {\bar \lambda}(\tau;\vm p(\cdot))$, $\timetau$ \cite[Sec. 2.4.2]{Vidyasagar}.
\end{proof}
Lemma \ref{lem:solvability_regularity_local_ocps} plays a central role for two reasons. First, it guarantees well-posedness of Algorithm \ref{alg:SBDP_cont_local_async}, ensuring existence and local uniqueness of solutions to the local OCPs~\eqref{eq:local_OCP_alg_async}. Second, it establishes Lipschitz continuity of these solutions w.r.t.\ the external trajectories. Based on Lemma \ref{lem:solvability_regularity_local_ocps}, the following convergence result can be shown.   
\begin{theorem} \label{thm:conv_SBDP}
Under Assumptions \ref{ass:asynchronism} and \ref{ass:coercivity}, there exists a horizon length $0<\bar T$  for every $\vm \xi \in \mathbb{X}_0$ and $\vm p^0(\cdot) \in \mathcal{B}_{r_1} (\vm p\inds(\cdot; \vm \xi))$ such that for any $T<\bar T$ there exists a constant $0<C<1$ such that the error $ e\indq(\vm \xi) := \|\vm p\indq(\cdot;\vm \xi) - \vm p\inds(\cdot;\vm \xi)\|_{b,\infty}$ of Algorithm~\ref{alg:SBDP_cont_local_async} converges R-linearly, i.e.,  
    \begin{align} \label{eq:R_linear_convergence}
       e\indq(\vm \xi) \leq \contr e^0(\vm \xi)\,, \quad q=0,1,\dots
    \end{align}
   holds and the iterates satisfy $ \vm p\indq(\cdot; \vm \xi) \in \mathcal{B}_{r_1} (\vm p\inds(\cdot; \vm \xi))$.
\end{theorem}
\begin{proof}
The proof follows from a direct application of \cite[Thm. 1]{Pierer}. We verify that all requirements are satisfied: i) Assumption \ref{ass:asynchronism} coincides with \cite[Ass. 1]{Pierer}; ii) the assumption of the existence of a locally unique centralized solution implies \cite[Ass. 2]{Pierer} where the uniqueness of $(\vm {\bar u}\inds(\cdot; \thisstate),\vm {\bar x}\inds(\cdot;\thisstate))$ together with the adjoint dynamics implies the uniqueness and boundedness of $\vm {\bar  \lambda}\inds(\cdot;\thisstate)$ as in Lemma \ref{lem:solvability_regularity_local_ocps}; iii) Let $e_i^q(\vm \xi):= \| \vm p_i\indq (\cdot; \vm \xi) -\vm p_i\inds (\cdot; \vm \xi)\|_{\infty}$, then \cite[Thm. 1]{Pierer} requires the estimate 
\begin{equation}\label{eq:async_contraction}
e_i\indqn(\vm \xi)\begin{cases}\leq C \max\limits_{j \in \mathcal{\bar N}_i}e_j\delayindex(\vm \xi)\,, &i =a(q)\\
= e_i\indq(\vm \xi)\,, &i \neq a(q)
\end{cases} 
\end{equation}
to hold which results by applying \cite[Lem. 2]{Pierer2} for $i = a(q)$ and realizing that for $i \neq a(q)$ the error remains unchanged. 
The assumptions of \cite[Lem. 2]{Pierer2} are met since i) Lemma \ref{lem:solvability_regularity_local_ocps} directly implies \cite[Ass. 3]{Pierer2}, ii) optimal and iterated trajectories are bounded on compact sets which is ensured by the fact that $(\vm { \bar u}(\cdot;\vm p(\cdot)), \vm { \bar x}(\cdot;\vm p(\cdot)),\vm { \bar \lambda}(\cdot;\vm p(\cdot)) ) $ are confined to compact sets according to Lemma \ref{lem:solvability_regularity_local_ocps}. Finally, the positive invariance of $\mathcal{B}_{r_1}(\vm p\inds(\cdot; \vm \xi))$ follows from the contraction estimate \eqref{eq:async_contraction} and Assumption~\ref{ass:asynchronism} which ensures uniformly bounded delays $d_{ij}\leq d$. To this end, taking the maximum over \eqref{eq:async_contraction} w.r.t.\ the delay interval yields $e^{q+1}(\vm \xi)\leq \max_{\ell\in\{q-d,\dots,q\}} C e^\ell(\vm \xi)$.
Because $C<1$ for all $T<\bar T$ and $e^0(\vm \xi)\leq r_1$, induction implies
$
e^q(\vm \xi)\leq e^0(\vm \xi)\leq r_1,\, \forall q\geq 0,
$
and therefore $\vm p^q(\cdot; \vm \xi)\in\mathcal B_{r_1}(\vm p\inds(\cdot; \vm \xi))$ for all iterations. Consequently, all iterates remain inside the region where Lemma~\ref{lem:solvability_regularity_local_ocps} and the contraction estimate remain valid.
\end{proof}
Theorem \ref{thm:conv_SBDP} establishes R-linear convergence of a-SBDP as long as the communication delays are bounded and no agents exit the optimization process. This is in contrast to the synchronous case, where Q-linear convergence can be guaranteed under the same assumptions \cite{Pierer2} and requires a careful consideration in the following stability analysis.  

%% file: content/stability_analysis.tex
\section{Stability analysis}
\label{sec:stability_analysis}
This section is concerned with establishing conditions under which the closed loop system resulting from applying the asynchronous DMPC control law \eqref{eq:DMPC_control_law} obtained after a limited number of global per-agent iterations $\bar q = \sum_{\agents} \bar q_i$ is stable. 
To this end, we make the following assumption concerning the value function $V(\cdot)$ associated with OCP \eqref{eq:ocp}. This assumption ensures that $V(\cdot)$ serves as a Lyapunov function for the closed loop system under the optimal MPC control law \eqref{eq:MPC_control_law} and fulfills certain regularity conditions. In this sense, it guarantees nominal stability of the centralized MPC.  
\begin{assumption} \label{ass:nominal_MPC_stability}
Assume that there exist constants $\alpha$, $a_1$, $a_2$, $a_3 \in \mathbb{R}_{>0}$ with $a_3<1$ such that for any $\vm \xi \in  \mathbb{X}_{\alpha} := \{\vm \xi \in \mathbb{X}_0 \, |\, V(\vm \xi) \leq \alpha\}$ the following holds
\begin{subequations} \label{eq:nominal_Lyapunov_decrease}
    \begin{align}
    a_1 \| \vm \xi\|_b^2 &\leq V(\vm \xi) \leq  a_2 \|\vm \xi\|_b^2 \label{eq:Lyapunov_1} \\
   V(\vm \Psi(\Delta t; \vm \xi, \vm { \bar u}\inds(\cdot;\vm \xi)))  &\leq  V(\vm \xi) - a_3 \|\vm \xi\|_b^2\,.\label{eq:Lyapunov_2}
\end{align}
\end{subequations}
In addition, it is assumed that $V(\cdot)$ is Lipschitz continuous on $ \mathbb{X}_0$ and twice continuously differentiable at $\vm \xi = \vm 0$ with $V(\vm \xi)> 0$ for $\vm \xi \in \mathbb{X}_0\setminus \mathbb{X}_\alpha$.
\end{assumption}
This assumption is formulated independently of a specific centralized OCP design to retain flexibility in the choice of the MPC formulation. In the context of DMPC, it can, for example, be enforced through suitable separable terminal costs \cite{Pierer2} or by invoking a cost controllability condition in combination with a sufficiently long prediction horizon \cite{Reble}. Note the Assumption \ref{ass:nominal_MPC_stability} is a pure local assumption. The (exponential) asymptotic stability of the origin with region of attraction $\mathbb{X}_{\alpha}$ under the optimal control law \eqref{eq:MPC_control_law} then follows from standard MPC-arguments. 
In similar spirit, the next assumption requires Lipschitz-like estimates of the change of optimal solutions of OCP \eqref{eq:central_OCP} w.r.t.\ perturbations of the initial state $\vm \xi \in \mathbb{X}_{\alpha}$ combined with a regularity assumption on the solution map $\vm \Psi(\cdot)$ of \eqref{eq:dynamics_ocp}. 
\begin{assumption} \label{ass:Lipschitz_solutions}
There exist constants $L_p, L_u \in \mathbb{R}_{>0}$ such that for any $\vm \xi \in \mathbb{X}_{\alpha}$ and  $\vm \xi^\prime \in \mathbb{X}_0$, the following holds
\begin{subequations}
    \begin{align}
    \|\vm p\inds(\cdot; \vm \xi^\prime) - \vm p\inds(\cdot; \vm \xi)\|_{b,\infty} &\leq L_p \| \vm \xi^\prime - \vm \xi\|_b \label{eq:Lipschitz_primal_dual_solution}\\
    \|\vm u\inds(\cdot; \vm \xi^\prime) - \vm u\inds(\cdot; \vm \xi)\|_{b,\infty} &\leq L_{u} \| \vm \xi^\prime - \vm \xi\|_b\,.
\end{align}
\end{subequations}
Moreover, we assume $\vm \Psi(\tau; \vm \xi, \vm u(\cdot)) \in \mathbb{X}_0$ for all $\vm \xi \in \mathbb{X}_\alpha$,  $\vm u(\tau) \in \mathbb{U}$, and $\timetausample$.  
\end{assumption}
The next Lemma provides an estimate between the system states resulting from applying the MPC and DMPC control laws \eqref{eq:MPC_control_law} and \eqref{eq:DMPC_control_law}, respectively, as well as a bound on the rate at which a state can change under~\eqref{eq:DMPC_control_law}.
\begin{lemma} \label{lem:diff_states}
Let Assumptions \ref{ass:coercivity} and \ref{ass:Lipschitz_solutions} hold. Then, for $\vm \xi \in \mathbb{X}_\alpha$ and $\vm p(\cdot) \in \mathcal{B}_{r_1}(\vm p\inds(\cdot; \vm \xi))$, there exist constants $b_1^\prime, b_2^\prime \in \mathbb{R}_{>0}$ such that the following estimates hold
\begin{subequations}
    \begin{align} 
    &\| \vm \xi_+ - \vm \xi_+\inds \|_b \leq b_1^\prime \| \vm p(\cdot) - \vm p\inds(\cdot; \vm \xi)\|_{b,\infty} \label{eq:bound_x_DMPC_x_MPC} \\
   &\| \vm \xi_+ - \vm \xi\|_b \leq b_2^\prime \| \vm \xi\|_b + b_1^\prime \| \vm p(\cdot) - \vm p\inds(\cdot; \vm \xi)\|_{b,\infty} \label{eq:diff_states}  
\end{align}
\end{subequations}
with $\vm \xi_+ = \vm \Psi(\Delta t; \vm \xi, \vm {\hat u}(\cdot; \vm \xi))$ and $ \vm \xi_+\inds =  \vm \Psi(\Delta t; \vm \xi, \vm u\inds (\cdot;\vm \xi))$.
\end{lemma}
\begin{proof}
To prove \eqref{eq:bound_x_DMPC_x_MPC}, consider the difference of the respective state trajectories generated by \eqref{eq:dynamics_ocp} for $\vm u\inds(\cdot; \vm \xi) = \vm u(\cdot; \vm p\inds(\cdot), \vm \xi)$ and $\vm {\hat u}(\cdot; \vm p(\cdot),\vm \xi)$. Then, we have
    \begin{align}
        & \| \vm \xi_+ - \vm \xi_+\inds \|_b = \nonumber\\
        &=\| \vm \Psi(\Delta t; \vm \xi,\vm {\hat u}(\cdot; \vm p(\cdot), \vm \xi)) - \vm \Psi(\Delta t; \vm \xi, \vm u\inds (\cdot;\vm \xi))\|_b \nonumber \\
        &\leq \!\! \int_0^{\Delta t} \!\!\|\vm f(\vm x(\tau), \vm {\hat u}(\tau; \vm p(\cdot), \vm \xi)) \!-\! \vm f(\vm x\inds(\tau), \vm u\inds(\tau; \vm \xi))\|_b \dd \tau \nonumber \\
        &\leq \Delta t L_{f,u} L \mathrm{e}^{L_{f,x} \Delta t} \|\vm p(\cdot) - \vm p\inds(\cdot; \vm \xi)\|_{b,\infty}
    \end{align}
where $L_{f,x},L_{f,u},L\in \mathbb{R}_{>0}$ follow from the differentiability of
$\vm f$ on $\mathbb X_0$ and Lemma
\ref{lem:solvability_regularity_local_ocps}. The last step uses
Gronwall's lemma, proving \eqref{eq:bound_x_DMPC_x_MPC} with $
b_1'=\Delta t\,L_{f,u}L\,e^{L_{f,x}\Delta t}$. For \eqref{eq:diff_states}, consider the centralized dynamics \eqref{eq:dynamics_ocp} under the DMPC control law \eqref{eq:DMPC_control_law}, i.e.,
    \begin{align} \label{eq:bound_xi2}
        & \| \vm \xi_+ - \vm \xi\|_b= \nonumber\\
       &= \|\vm \Psi(\Delta t; \vm \xi,\vm {\hat u}(\cdot; \vm p(\cdot), \vm \xi)) - \vm \xi \|_b =: \| \vm r(\Delta t)\|_b \nonumber\\
        &\leq \int_{0}^{\Delta t } \| \vm f( \vm x(\tau) + \vm r(\tau), \vm {\hat u}(\cdot; \vm p(\cdot), \vm \xi))\|_b\, \dd \tau \nonumber \\
        &\leq \! \Delta t \mathrm{e}^{L_{f,x} \Delta t}  ( L_{f,x}\| \vm \xi \|_b \!+\! L_{f,u}\| \vm {\hat u}(\cdot; \vm p(\cdot), \vm \xi)\|_{b,\infty})\,,
    \end{align}
where the last inequality again follows from Gronwall's lemma. Moreover, we have
\begin{multline} \label{eq:bound_u}
\| \vm {\hat u}(\cdot; \vm p(\cdot), \vm \xi) - \vm u\inds(\cdot; \vm \xi) + \vm u\inds(\cdot; \vm \xi) - \vm u\inds(\cdot; \vm 0)\|_{b,\infty} \\
    \leq L\| \vm p(\cdot) -  \vm p\inds(\cdot; \vm \xi)\|_{b,\infty} + L_{u}\|\vm \xi\|_b\,,
\end{multline}
using $\vm u\inds(\cdot; \vm 0) = \vm 0$.
Substituting \eqref{eq:bound_u} into \eqref{eq:bound_xi2} yields \eqref{eq:diff_states} with $b_2^\prime = \Delta t (L_{f,x}+L_{f,u} L_u) \mathrm{e}^{L_{f,x} \Delta t} $.
\end{proof}
Lemma \ref{lem:diff_states} and Assumption \ref{ass:nominal_MPC_stability} allow us to quantify the effect of the error $\| \vm p(\cdot) - \vm p\inds (\cdot, \vm \xi)\|_{b,\infty}$ on the nominal Lyapunov decrease condition \eqref{eq:Lyapunov_2}.  
\begin{lemma} \label{lem:perturbed_Lyapunov}
Let Assumptions \ref{ass:coercivity}, \ref{ass:nominal_MPC_stability}, and \ref{ass:Lipschitz_solutions} hold and let $a:= \frac{a_3}{a_2}$. Then, there exists a constant $b_1 \in \mathbb{R}_{>0}$ such that for any $\vm \xi \in \mathbb{X}_{\alpha}$ and $\vm p\indq(\cdot;\vm \xi) \in \mathcal{B}_{r_1}(\vm p\inds(\cdot; \vm \xi))$, we have 
    \begin{align}\label{eq:perturbed_Lyapunov}
        V(\vm \xi_+) &\leq (1-a) V(\vm \xi) +  b_1 e\indq(\vm \xi)\,.
    \end{align}
\end{lemma}
\begin{proof}
Under Assumption \ref{ass:nominal_MPC_stability} there exists a constant $L_V \in \mathbb{R}_{>0}$ such that $ |V(\vm \xi^\prime) - V(\vm \xi)|\leq L_V \| \vm \xi^\prime - \vm \xi\|_b$ holds for any $\vm \xi^\prime, \vm \xi \in \mathbb{X}_0$ which implies
    \begin{align}
        V(\vm \xi_+)&\leq V(\vm \xi_+\inds) + L_V \|\vm \xi_+ - \vm \xi_+\inds\|_b \nonumber\\
        &\leq\! (1 - a) V(\vm \xi) + b_1 e\indq(\vm \xi) 
    \end{align}
with Assumption \ref{ass:nominal_MPC_stability} and \eqref{eq:bound_x_DMPC_x_MPC} in Lemma \ref{lem:diff_states} being applied. This proves \eqref{eq:perturbed_Lyapunov} with $b_1 = L_Vb_1^\prime$.
\end{proof}
The following result generalizes the convergence result~\eqref{eq:R_linear_convergence} in Theorem \ref{thm:conv_SBDP} to the case where the evolution of the initial state $\vm \xi$ is governed by the DMPC control law~\eqref{eq:DMPC_control_law} and gives an estimate of the error $e^q(\vm \xi_+)$ in terms of the previous error $e^q(\vm \xi)$ and the value function $V(\thisstate)$. 

\begin{lemma} \label{lem:perturbed_contraction}
Let Assumptions \ref{ass:asynchronism}, \ref{ass:coercivity}, \ref{ass:nominal_MPC_stability}, and \ref{ass:Lipschitz_solutions} hold. Then, there exist constants $\bar \alpha, r_2, b_2, b_3 \in \mathbb{R}_{>0}$ such that for any $\vm \xi \in \mathbb{X}_{\bar \alpha}$ and $\vm p\indq(\cdot;\vm \xi) \in \mathcal{B}_{r_2}(\vm p\inds(\cdot; \vm \xi))$, we have
\begin{align}\label{eq:perturbed_contraction}
 e\indq(\vm \xi_+) &\leq  \contr (b_2 \sqrt{V(\vm\xi)} + b_3  e\indq(\vm \xi))\,.
\end{align}
\end{lemma}
\begin{proof}
For $\bar \alpha:= \min\{ \frac{r_1}{L_p b_2^\prime \sqrt{a_1}}, \alpha\}>0$ and $r_2 :=r_1 - L_p b_2^\prime \sqrt{a_1} \bar \alpha>0$, we get that $e^0(\vm \xi_+) \in \mathcal{B}_{r_1}(\vm p\inds(\cdot; \vm \xi_+))$ and can apply \eqref{eq:R_linear_convergence} from Theorem~\ref{thm:conv_SBDP} at $\vm \xi_+$, i.e., 
\begin{align} \label{eq:perturbed_contraction_deriv}
        e\indq(\vm \xi_+)&\leq \contr (e^q(\vm \xi) + \| \vm p\inds(\cdot; \vm \xi_+) - \vm p\inds(\cdot; \vm \xi)\|_{b,\infty}) \nonumber\\
        & \leq \contr ( e^q(\vm \xi) + L_p \|\vm \xi_+ - \vm \xi \|_b)\,,
    \end{align}
using the warm start and \eqref{eq:Lipschitz_primal_dual_solution} in
Assumption~\ref{ass:Lipschitz_solutions}. Combining \eqref{eq:perturbed_contraction_deriv} with
Lemma~\ref{lem:diff_states} and \eqref{eq:Lyapunov_1} from
Assumption~\ref{ass:nominal_MPC_stability} yields \eqref{eq:perturbed_contraction} with
$b_2=L_p b_2'\sqrt{a_2}$ and $b_3=1+L_p b_3'$. 
\end{proof}
The error bounds given in Lemmas \ref{lem:perturbed_Lyapunov} and \ref{lem:perturbed_contraction} define a discrete-time dynamical system which will be used to infer stability properties of the closed loop system. Let $w_1[k]: = V(\vm x_k)$ and $w_2[k]:= e\indq(\vm x_k)$ denote the value function and optimization error at MPC step $k$, respectively. Then, by \eqref{eq:perturbed_Lyapunov} and \eqref{eq:perturbed_contraction}, the system
\begin{subequations} \label{eq:discrete_time_dynamics}
    \begin{align}
 w_1[k+1] &= (1- a) w_1[k] + b_1 w_2[k] \\
  w_2[k+1] &=  \contr (b_2 \sqrt{w_1[k]} + b_3 w_2[k])\,,
    \end{align}
\end{subequations}
with $w_1[0] = \bar \alpha$ and $w_2[0] = \contr r_2$, provides an upper bound on the evolution of the quantities $V(\vm x_k)$ and $e^q(\vm x_k)$ along the closed loop trajectory. The next lemma establishes a positive invariant set for the closed loop extended state $\vm z[k]:= (\vm x_k, \vm p\indq(\cdot; \vm x_k))$ using \eqref{eq:discrete_time_dynamics}. 
\begin{lemma} \label{lem:invariant_set}
Let Assumptions \ref{ass:asynchronism}, \ref{ass:coercivity}, \ref{ass:nominal_MPC_stability} and \ref{ass:Lipschitz_solutions} hold and let $ r:= \min\{\frac{\bar \alpha a_3}{b_1a_1}, r_2\}$. Moreover, define the set
\begin{multline} \label{eq:invariant_set}
    \Gamma:=\{ \vm \xi \in \mathbb{R}^{n},\, \vm p(\cdot) \in C([0,T]; \mathbb{R}^p) \, | \, \\ V(\vm \xi) \leq \bar \alpha,\, \| \vm p(\cdot) - \vm p\inds(\cdot; \vm \xi)\|_{b,\infty}\leq r\}\,.
\end{multline}
If the global number of per-agent iterations satisfies 
    \begin{align} \label{eq:invariant_set_num_iterations}
        q > q_1\!:=\! \bigg(1 + \log_{C}\bigg(\frac{r}{b_2 \sqrt{\bar \alpha}+ b_3 r}\bigg)\bigg)(s+d)\,,
    \end{align}
then $\Gamma $ is positive invariant for the closed loop system in the sense of $ \vm z[k] \in \Gamma$ implies $ \vm z[k+1] \in \Gamma$, $k=0,1,\dots$
\end{lemma}

\begin{proof}
Let $\vm z[k] \in \Gamma$. Then, $\vm x_k \in \mathbb{X}_{\bar \alpha}$ and  $\vm p\indq(\cdot; \vm x_k) \in \mathcal{B}_{r_1}(\vm p\inds(\cdot; \vm x_k)) $ which allows us to apply Lemma \ref{lem:perturbed_Lyapunov}
    \begin{align}
        V(\vm x_{k+1}) &\leq (1- a) V(\vm x_k) + b_1 e\indq(\vm x_k) \nonumber\\
        &\leq (1 - a) \bar \alpha + \bar \alpha a = \bar \alpha 
    \end{align}
    which implies $\vm x_{k+1} \in \mathbb{X}_{\bar \alpha}$. Moreover, Lemma \ref{lem:perturbed_contraction} yields
   \begin{align}
    e\indq(\vm x_{k+1})  &\leq  \contr (b_2 \sqrt{V(\vm x_k)} + b_3  e\indq(\vm x_{k}))\nonumber\\
  &\leq  \contr (b_2 \sqrt{\bar \alpha} + b_3 r)
  \leq r\,,
\end{align}
where \eqref{eq:invariant_set_num_iterations} is used, implying $e\indq(\vm x_{k+1}) \leq r$. 
\end{proof}
The dynamics \eqref{eq:discrete_time_dynamics} are difficult to analyze with standard tools as they are not Lipschitz continuous at the origin. However, we can rewrite Lemma \ref{lem:perturbed_Lyapunov} in terms of $\sqrt{V(\vm\xi)}$ instead of $V(\vm \xi)$ to obtain a system which is linear.
\begin{lemma} \label{lem:Lypunov_decrease_2}
    Let Assumptions \ref{ass:coercivity}, \ref{ass:nominal_MPC_stability}, and \ref{ass:Lipschitz_solutions} hold. Then, there exists a constant $\bar b_1 \in \mathbb{R}_{>0}$ such that for any $\vm \xi \in \mathbb{X}_{\alpha}$ and $\vm p\indq(\cdot; \vm \xi) \in \mathcal{B}_{r_1}(\vm p\inds(\cdot; \vm \xi))$, we have 
\begin{align}\label{eq:Lyapunov_decrease_2}
      \sqrt{V(\vm \xi_+)} \leq \sqrt{(1- a)} \sqrt{V(\vm \xi)}+ \bar b_1 e\indq(\vm \xi)\,.
\end{align}
\end{lemma}
\begin{proof}
Assumption \ref{ass:nominal_MPC_stability} implies that $\sqrt{V(\vm \xi)}$ is Lipschitz on $\mathbb{X}_0$ with Lipschitz constant $\bar L_V =\max\{\frac{L_V}{2\sqrt{a_1}},\frac{L_V}{2\sqrt{\bar \delta}}\}$ with $\bar \delta = \min_{\vm \xi \in \mathbb{X}_0 \setminus \mathbb{X}_\alpha} V(\vm \xi)>0$ such that $ | \sqrt{V(\vm \xi^\prime)} - \sqrt{ V(\vm \xi)}| \leq \bar L_V \| \vm \xi^\prime - \vm \xi\|_b$ holds for any $\vm \xi^\prime, \vm \xi \in \mathbb{X}_0$. The rest of the proof directly follows along the lines of the proof of Lemma \ref{lem:perturbed_Lyapunov} with $\bar b_1 = b_1^\prime \bar L_V$. 
\end{proof}
Utilizing the Lyapunov decrease estimate \eqref{eq:Lyapunov_decrease_2} instead of \eqref{eq:perturbed_Lyapunov}, consider the recursion defined by \eqref{eq:perturbed_contraction} and \eqref{eq:Lyapunov_decrease_2} which can be cast as the linear discrete-time system 
\begin{subequations} \label{eq:auxiliary_discrete_time_dynamics}
    \begin{align}
 \hat w_1[k+1] &= \sqrt{(1- a)} \hat w_1[k] + \bar b_1 \hat w_2[k] \\
  \hat w_2[k+1] &=  \contr (b_2 \hat w_1[k] + b_3 \hat w_2[k])
    \end{align}
\end{subequations}
with states $\hat w_1[k] = \sqrt{V(\vm x_k)}$ and $\hat w_2[k] = e\indq(\vm x_k)$ as well as initial conditions $\hat w_1[0] = \sqrt{\bar \alpha}$ and $\hat w_2[0] = \contr r$. The next theorem exploits the linearity of \eqref{eq:auxiliary_discrete_time_dynamics} to derive conditions for the closed loop stability. 
\begin{theorem} \label{thm:DMPC_stability}
Let Assumptions \ref{ass:asynchronism}, \ref{ass:coercivity}, \ref{ass:nominal_MPC_stability}, and \ref{ass:Lipschitz_solutions} hold. Then, for $\vm z[0] \in \Gamma$ and the global number of per-agent iterations satisfying $ \bar q > \max\{q_1,q_2\}$ with 
    \begin{align}\label{eq:bound2_q}
         q_2 := \bigg(1 + \log_{C}\bigg(\frac{\tilde a}{\tilde a b_3 + \bar b_1b_2}\bigg)\bigg)(s+d)
    \end{align}
and  $\tilde a := 1 - \sqrt{(1 - a)} \in (0,1)$, the origin of the closed loop system is exponentially asymptotically stable under the control law \eqref{eq:DMPC_control_law}. Moreover, the value function $V(\vm x_k)$ and optimization error $e\indq(\vm x_k)$ decay exponentially.  
\end{theorem}
\begin{proof} 
We first establish exponential stability of
\eqref{eq:auxiliary_discrete_time_dynamics} under \eqref{eq:bound2_q}
and then use this result to prove exponential stability of the closed
loop under \eqref{eq:DMPC_control_law}.
Let $\vm A(q)$ denote the system matrix of
\eqref{eq:auxiliary_discrete_time_dynamics} and
$ p(\lambda) = \lambda^2-\mathrm{tr}(\vm A(q))\lambda+\det(\vm A(q))$
its characteristic polynomial. Since all entries of $\vm A(q)$ are
strictly positive, the Perron-Frobenius theorem implies that
$\lambda_1=\spec(\vm A(q))$ is a simple positive eigenvalue and
$\lambda_2<\lambda_1$. Hence, $\vm A(q)$ is Schur stable if and only if
$\lambda_1<1$, which is equivalent to $p(1)=(1-\lambda_1)(1-\lambda_2)>0$.
Since $\lambda_2<\lambda_1<1$ implies $(1-\lambda_2)>0$, Schur
stability follows from $ 1-\mathrm{tr}(\vm A(q))+\det(\vm A(q))>0$ ,
which yields \eqref{eq:bound2_q}. Therefore, the origin of
\eqref{eq:auxiliary_discrete_time_dynamics} is exponentially stable.
Consequently, there exist constants $c_0>0$ and $c_1\in(0,1)$ such that $
\|\hat{\vm w}[k]\| \leq c_0c_1^k\|\hat{\vm w}[0]\|$ with
$\hat{\vm w}
=
[\hat w_1,\hat w_2]\trans$. Thus, it holds that $
w_1[k]
\leq
c_0^2c_1^{2k} \| [\sqrt{w_1[0]},\,w_2[0]]\trans \|^2$
and
$ w_2[k] \leq c_0c_1^k \| [\sqrt{w_1[0]},\,w_2[0]]^\top \|$,
showing exponential decay of the value function and optimization error,
respectively. Consider now the closed loop system $
\dot{\vm x}(t)
=
\vm f(\vm x(t),\vm{\hat u}(t_k+\tau;\vm x_k))$, $\tau\in[0,\Delta t)$, $ t\in[t_k,t_{k+1}) $. 
Using \eqref{eq:Lyapunov_1} yields $
\|\vm x_k\|
\leq
\sqrt{M}a_1^{-1/2}\sqrt{w_1[k]}
\leq
c_2c_1^k\|\vm x_0\|$, with $c_2 = \sqrt{M}a_1^{-1/2}c_0(1+\|w_2[0]\|)$.
Standard sampled-data arguments then imply the existence of
$\gamma_1,\gamma_2>0$ such that$
\|\vm x(t)\|
\leq
\gamma_1\|\vm x_0\|e^{-\gamma_2 t}$, $t\geq 0$, establishing exponential stability of the closed loop system. 
\end{proof}

Theorem \ref{thm:DMPC_stability} establishes exponential stability of the asynchronous DMPC scheme provided that the number of global per-agent iterations $\bar q$ is sufficiently large. Since $\bar q$ corresponds to the total number of agent activations within one MPC step, there is a certain flexibility in how these activations are distributed among agents, e.g.,  $\bar q_i = \frac{\bar q}{M}$. 

%% file: content/validation.tex
\section{Simulative and experimental validation} \label{sec:validation}
To illustrate Theorem \ref{thm:DMPC_stability} and validate the asynchronous DMPC scheme, we consider the stabilization of
\begin{align} \label{eq:example_system}
    \dot x_i(t) = u_i(t) + \sum_{\neighs} \sin(x_j(t))\,, \quad x_i(0) = x_{i,0}
\end{align}
at the (unstable) origin with stage costs $l_i(x_i,u_i)=x_i^2+0.1 u_i^2$, input constraints $u_i(t)\in[-2,2]$, and $M=3$. The analysis consists of i) estimating the constants in Theorem~\ref{thm:DMPC_stability} to compute $\bar q$, and ii) a hardware-in-the-loop (HIL) validation where asynchronism naturally occurs.

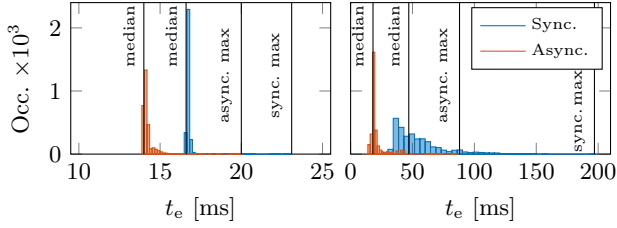
\begin{figure}[h]
\centering
\input{content/comp_times.tex}
\caption{Histogram of the execution time $t_{\mathrm{e}}$ per MPC step for (a)synchronous SBDP with a fixed number of $q_i=7$ per-agent iterations within an Ethernet (left) and Wi-Fi (right) network. For all four combinations, the closed loop cost $J_{cl}\approx 0.55$ is basically identical, illustrating that the a-SBDP achieves the same control performance in less time.} \label{fig:hist_comp_times}
\end{figure}

\textbf{ i) Numerical validation:} Five representative initial conditions $|x_{i,0}|\leq 1$ are considered and the system \eqref{eq:example_system} is controlled using the optimal MPC law \eqref{eq:MPC_control_law} and the inexact DMPC law \eqref{eq:DMPC_control_law}. Without terminal costs, nominal stability is ensured via a sufficiently long horizon together with exponential controllability in the sense of \cite[Ass.~14]{Reble}. The required candidate control trajectories are chosen as open loop controls from each initial state over the simulation horizon, verifying the assumption for all cases. The resulting worst-case MPC contraction rate is $\alpha_{\mathrm{mpc}}=0.083$ (cf. \cite[Thm. 1]{Reble}) for $T=3$ and $\Delta t=0.1$.
The constants in Assumptions~\ref{ass:nominal_MPC_stability} and \ref{ass:Lipschitz_solutions} are obtained by evaluating $V(\vm x_k)$, $\vm u\inds(\cdot;\vm x_k)$, $\vm p\inds(\cdot;\vm x_k)$, and $\|\vm x_k\|_b$ along closed-loop trajectories, yielding $
a_1=0.31$, $a_2=0.96$, $a_3=0.12$, $L_p=1.62$, $L_u=4.62$, $
L_{f,u}=1.00$, $\quad L_{f,x}=2.00$, $\bar L_V=0.27$.
The constants $C=0.58$ (Theorem~\ref{thm:conv_SBDP}) and $L=16.65$ (Lemma~\ref{lem:solvability_regularity_local_ocps}) are computed from five synchronous iterations compared against the solution at $\vm x_k$. Assumption~\ref{ass:coercivity} holds with, e.g., $\rho_i=0.1$, since $Q_i(\tau)= 1$, $\timetau$. Finally, the resulting constants in Theorem \ref{thm:DMPC_stability} are $\bar b_1=0.55$, $b_2=1.29$, $b_3=4.30$, and $\tilde a=0.06$. Thus, for a synchronous execution $(s=1,\, d=0)$, $\bar q=7>q_2=6.01$ iterations guarantee stability, whereas in the asynchronous case the bound scales with the degree of asynchronism and becomes $q_2=6.01(s+d)$. We empirically verify that $\bar q=7>q_1$ as $\vm z[k]\in\Gamma$ due to $V(\vm x_k)\leq V(\vm x_0)$ and $ e\indq(\vm x_k) \leq e^q(\vm x_0) $.

\textbf{ii) HIL experiment:} The asynchronous DMPC scheme is implemented in C\texttt{++} and executed on a hardware network of Raspberry Pi~3B+ connected via Ethernet and Wi-Fi for $\vm x_0=[0.75,0.5,0.25]\trans$. A synchronous execution enforces the agents to wait until all neighboring data has arrived, while the asynchronous version directly proceeds with available neighbor data.
Communication is implemented over UDP using the \texttt{LCM} library based on a publish-subscribe paradigm. This matches the proposed framework, which tolerates message loss and exploits multicast communication in Algorithm~\ref{alg:SBDP_cont_local_async}. The local OCPs \eqref{eq:local_OCP_alg_async} are solved via a damped forward-backward sweep. Figure~\ref{fig:hist_comp_times} shows per-agent execution time histograms for $\bar q_i=7$ per-agent iterations for $10^3$ MPC steps. The asynchronous scheme reduces median and maximum execution times by approximately $15\,\%$ respective $19\,\%$ for Ethernet and $61\,\%$ respective $63\,\%$ for Wi-Fi, while only increasing the closed loop cost
$
J_{\mathrm{cl}}=\int_0^{T_{\mathrm{sim}}} l_i(x_i(t),u_i(t))\,\mathrm dt
$
by around $0.2\,\%$ for both Ethernet and Wi-Fi. In addition, approximately only $80\, \%$ of the MPC steps in the synchronous Wi-Fi experiment were successful due to package losses combined with UDP-based communication. In case of an unsuccessful MPC step, the simulation is restarted. This indicates that the asynchronous DMPC scheme also increases robustness w.r.t.\ lossy communication.
After the execution we estimate $\bar q$ from the Ethernet data. To reduce conservatism, we use $e\indq(\vm \xi) \leq C^{-1}\hat C^q e^0(\vm \xi)$, cf. \eqref{eq:R_linear_convergence}, with $\hat C = C^{\frac{1}{s+d}}$ and determine $\hat C$ directly from the asynchronous execution rather than of estimating the constants $s$ and $d$. This yields $q_1 =20.43 < \bar q=21 = \sum_{\agents} \bar q_i$, thereby also verifying Theorem \ref{thm:DMPC_stability} for the asynchronous execution. The code for both the numerical validation and the HIL experiment is available at \url{https://github.com/maximilianpierervonesch/async_sbdp_dmpc}. 
%.

%% file: content/comp_times.tex
\begin{tikzpicture}
\begin{groupplot}[
    group style={
        group size = 2 by 1,
        horizontal sep= 7.5pt,
    },
    width=0.595\columnwidth,
    height=0.45\columnwidth,
    xlabel={$t_{\mathrm{e}}$ $[\si{\milli\second}]$},
    legend pos=north east,
    grid=none,
    scaled ticks=false,
    xlabel style={font=\small},
    ylabel style={font=\small},
    legend style={font=\tiny},
    tick label style={font=\small},
    legend cell align=left,
]

% First plot
\nextgroupplot[
        ylabel={Occ. $\times 10^3$},
        ymin=0,
        ymax=2400,
        xmin = 9.5,
        xmax = 25.5,
        yticklabels = {,0,1,2}
]

\addplot[
    ybar interval,
    fill=mycolor1,
    draw=mycolor1,
    fill opacity=0.5,
] table[
    x=bin_center_sync,
    y=sync,
    col sep=comma
]{data/histogram_data_ethernet.csv};

\addplot[
    ybar interval,
    fill=mycolor2,
    draw=mycolor2,
    fill opacity=0.5,
] table[
    x=bin_center_async,
    y=async,
    col sep=comma
]{data/histogram_data_ethernet.csv};

\addplot[
    black,
    thin,
] coordinates {(16.6,-0) (16.6,2500)}
node[pos=0.75, rotate = 90, anchor=south] {\tiny median};

\addplot[
    black,
    thin
] coordinates {(14,-0) (14,2500)}  node[pos=0.75, rotate = 90, anchor=south] {\tiny median};

\addplot[
    black,
    thin,
] coordinates {(20,-0) (20,2500)} node[pos=0.5, rotate = 90, anchor=south] {\tiny async. max};

\addplot[
    black,
    thin
] coordinates {(23.1,-0) (23.1,2500)} node[pos=0.5, rotate = 90, anchor=south] {\tiny sync. max};

% Second plot
\nextgroupplot[
    ymin=0,
    ymax=2400,
     xmin = 0,
     xmax = 210,
     yticklabels = {},
]
\addplot[
    ybar interval,
    fill=mycolor1,
    draw=mycolor1,
    fill opacity=0.5,
] table[
    x=bin_center_sync,
    y=sync,
    col sep=comma
]{data/histogram_data_wifi.csv};

\addplot[
    ybar interval,
    fill=mycolor2,
    draw=mycolor2,
    fill opacity=0.5,
] table[
    x=bin_center_async,
    y=async,
    col sep=comma
]{data/histogram_data_wifi.csv};

\addplot[
    black,
    thin,
] coordinates {(18.1,-0) (18.1,2500)}
node[pos=0.75, rotate = 90, anchor=south] {\tiny median};

\addplot[
    black,
    thin
] coordinates {(47,-0) (47,2500)}  node[pos=0.75, rotate = 90, anchor=south] {\tiny median};

\addplot[
    black,
    thin
] coordinates {(88,-0) (88,2500)} node[pos=0.5, rotate = 90, anchor=south] {\tiny async. max};

\addplot[
    black,
    thin,
] coordinates {(197,-0) (197,2500)} node[pos=0.275, rotate = 90, anchor=south] {\tiny sync. \hspace{-4pt} max};

\legend{Sync., Async.}

\end{groupplot}
\end{tikzpicture}

%% file: content/conclusions.tex
\section{Conclusions} \label{sec:conclusion}
This paper presents a cooperative DMPC scheme based on a-SBDP for the setpoint stabilization of nonlinear systems. Asynchronous execution enables agents to perform iterations independently using outdated information, providing a practical framework for DMPC which is able to deal with, e.g., lossy communication or stragglers. Exponential stability of the closed loop system is established under a nominally stable MPC scheme and partial asynchronism which allows arbitrary but bounded communication delays between agents. 

The trade-off inherent to asynchronous optimization, i.e., slower convergence vs. reduced execution time, is quantified in the context of MPC with a HIL-experiment, where the execution time is reduced by up to $19\,\%$ for Ethernet and $63\, \%$ for Wi-Fi with essentially unchanged control performance. Furthermore, the asynchronous execution increases the robustness of the optimization process. Future research concerns more extensive experimental validations and extensions to state constraints.